\documentclass{amsart}
\usepackage{graphicx} 
\usepackage{amsfonts}
\usepackage{amssymb}
\usepackage{color}
\usepackage{amsmath}
\usepackage{amsthm}
\usepackage{multicol}
\usepackage{mathtools}
\usepackage{tikz,tikz-cd}
\usepackage[foot]{amsaddr}
\usepackage[colorlinks=true,linkcolor=blue,citecolor=blue]{hyperref}

\theoremstyle{plain}
\newtheorem{theorem}{Theorem}[section]
\newtheorem{lemma}[theorem]{Lemma}
\newtheorem{proposition}[theorem]{Proposition}
\newtheorem{corollary}[theorem]{Corollary}

\theoremstyle{definition}
\newtheorem{definition}[theorem]{Definition}
\newtheorem{remark}[theorem]{Remark}

\begin{document}
\title[Functional monadic ortholattices]{Functional monadic ortholattices and locally finite $\sigma$-free polyadic ortholattices}
\author{Chun-Yu Lin $^{1,2}$, Joseph McDonald $^{1,3}$}
\address{$^1$ Institute of Computer Science, Czech Academy of Science, Prage, Czech Republic}
\address{$^2$ Department of Logic, Charles University, Prague, Czech Republic}
\address{$^3$ Department of Philosophy, University of Alberta, Edmonton, Canada}
\email{lin@cs.cas.cz}
\email{jsmcdon1@ualberta.ca}
\date{March 2025}
\maketitle
\begin{abstract}
    In this paper, we show that every monadic ortholattice is isomorphic to a functional one, thereby resolving a recent question posed by Harding. We then study certain substitution-free reducts of the polyadic ortholattices, which we call \emph{locally finite $\sigma$-free polyadic ortholattices}, and provide an analogous functional representation result. 
\end{abstract}

\section{Introduction}

Monadic algebras were introduced by Halmos \cite{halmos1} as algebraic models of the classical predicate calculus in a single variable. A \emph{monadic algebra} is a Boolean algebra $\langle B;\wedge,\vee,\neg,0,1\rangle$ equipped with a closure operator $\exists\colon B\to B$, known as a \emph{quantifier}, whose closed elements form a Boolean sub-algebra. A standard construction of a monadic algebra involves starting with a non-empty set $X$, a complete Boolean algebra $B$, and taking the function space $B^X$. Then, by defining the relevant Boolean algebra operations point-wise and defining the associated quantifier by: \[(\exists f)(x)=\bigvee\{f(x):x\in X\}\] for each $f\in B^X$, one obtains a monadic algebra, known as the \emph{full functional monadic algebra}. A monadic algebra is then called a \emph{functional monadic algebra} provided it is a subalgebra of a full functional monadic algebra. Halmos \cite{halmos1} proved a representation theorem for monadic algebras by demonstrating that every monadic algebra is isomorphic to a functional one. Closely related to the monadic algebras are the polyadic and cylindric algebras. Polyadic algebras were introduced by Halmos \cite{halmos1} as algebraic models of the full classical predicate calculus without equality. A polyadic algebra is a quadruple $\langle B;I,\exists,\sigma\rangle$ consisting of a Boolean algebra $B$, a set $I$, a function $\exists$ from the powerset $\wp(I)$ of $I$ to all Boolean endomorphisms $\text{End}(B)$ on $B$, and a function $\sigma$ from the function space $I^I$ to $\text{End}(B)$, all of which is subject to certain conditions. Cylindric algebras on the other hand were introduced by Henkin and Tarski \cite{henkin} and by Henkin, Monk, and Tarski \cite{tarski1,tarski2} as algebraic models of the classical predicate calculus with equality. A cylindric algebra $\langle B;(\exists_i)_{i\in I},(\delta_{i,k})_{i,k\in I}\rangle$ consists of a Boolean algebra $B$, a family $(\exists_i)_{i\in I}$ of pairwise commuting quantifiers on $B$, and a family of $(\delta_{i,k})_{i,k\in I}$ of constants, known as the \emph{diagonal elements}, satisfying certain conditions. Functional counterparts for both these classes of algebras are defined somewhat analogously to the case of monadic algebras and functional representation theorems for them are known (see \cite{halmos1, henkin1}).  

The first aim of this paper is to study the functional representation theory of monadic ortholattices. Ortholattices form generalizations of Boolean algebras in the sense that their bounded lattice reducts are not in general distributive. Ortholattices are interesting within the setting of non-classical logic as they form algebraic models of orthologic \cite{birkhoff, chiara, goldblatt2}. Monadic ortholattices were first introduced by Janowitz \cite{janowitz}, and then later studied by Harding \cite{harding1} as well as Harding, McDonald, and Peinado \cite{harding2}. In the second section, we introduce functional monadic ortholattices as generalizations of the functional monadic algebras and demonstrate that every monadic ortholattice is isomorphic to a functional one. This resolves a recent question posed by Harding \cite{harding1}. Our methods for proving this functional representation theorem use those developed by Bezhanishvili and Harding \cite{bezhanishvili} in the setting of functional monadic Heyting algebras, and make use of the well-known facts that the variety of ortholattices satisfy certain amalgamation properties \cite{bruns, miyazaki}, and are closed under MacNeille completions \cite{maclaren} and hence admit of regular completions \cite{day}. We then study diagonal-free ($\delta$-free) cylindric ortholattices. A $\delta$-free cylindric ortholattice is a special reduct of a cylindric ortholattice \cite{harding1, mcdonald}, and consists of an ortholattice equipped with a family of pairwise commuting quantifiers. As a consequence of our functional representation theorem for monadic ortholattices, it is shown that every $\delta$-free cylindric ortholattice is isomorphic to a functional one. In the third section, we introduce certain polyadic ortholattices, which we call \emph{locally finite substitution-free ($\sigma$-free) polyadic ortholattices}. These are certain generalizations of the locally finite polyadic algebras without substitution and can be viewed as the polyadic variants of the locally finite $\delta$-free cylindric ortholattices. We then show that the correspondence between locally finite $\sigma$-free polyadic ortholattices and locally finite $\delta$-free cylindric ortholattices is one-to-one. We conclude by demonstrating that every locally finite $\sigma$-free polyadic ortholattice is isomorphic to a functional one through this correspondence.

\section{Functional Monadic Ortholattices}
We begin by introducing general ortholattices and monadic ortholattices. For more details on the former, consult \cite{birkhoff2} and for the latter, consult \cite{harding1}. Our results later on in this section will make use of algebraic concepts such as V-formation, super-amalgamation, MacNeille completion, and regular completion. For more details on related concepts, consult Gr\"atzer \cite{gratzer}.  
\begin{definition}\label{ortholattice}
    An \emph{ortholattice} is a bounded lattice $\langle A;\wedge,\vee,0,1\rangle$ equipped with an operation $^{\perp}\colon A\to A$, known as an \emph{orthocomplementation}, satisfying:
    \begin{enumerate}
        \item $a\wedge a^{\perp}=0$, $a\vee a^{\perp}=1$; 
        \item $a\leq b\Rightarrow b^{\perp}\leq a^{\perp}$; 
        \item $a^{\perp\perp}=a$. 
    \end{enumerate}
\end{definition}

Conditions 1 through 3 of Definition \ref{ortholattice} amount to asserting that the operation of orthocomplementation is an order-inverting involutive complementation. We note that every ortholattice satisfies De Morgan's distribution laws and that an ortholattice $A$ is a Boolean algebra iff $A$ is distributive. Ortholattices are well known to be definable via a finite set of equations and hence form a variety. We denote by $\mathbb{OL}$ the variety of ortholattices.

\begin{definition}
    A \emph{monadic ortholattice} is an ortholattice $\langle A;\wedge,\vee, ^{\perp},0,1\rangle$ equipped with an additional operator $\exists\colon A\to A$, known as a \emph{quantifier}, satisfying:  
    \begin{enumerate}
        \item $\exists(a\vee b)=\exists a\vee\exists b$
        \item $\exists 0=0$
        \item $\exists\exists a=\exists a$
        \item $a\leq\exists a$
        \item $\exists(\exists a)^{\perp}=(\exists a)^{\perp}$
    \end{enumerate}
\end{definition}
Any quantifier on an ortholattice $A$ may be viewed as a closure operator whose \emph{closed elements}, i.e., $\{a\in A:a=\exists a\}$, form a sub-ortholattice of $A$.  It is obvious that $\exists\colon A\to A$ is monotone for any ortholattice $A$. 

\begin{remark}
We note that monadic ortholattices can be equivalently defined by an interior operator $\forall\colon A\to A$ on an ortholattice $A$ whose \emph{open elements}, i.e., $\{a\in A:a=\forall a\}$, form a sub-ortholattice of $A$, by defining $\forall a:=(\exists a^{\perp})^{\perp}$.  
\end{remark}

We now introduce the full functional monadic ortholattices. 

\begin{definition}\label{full funcitonal mol}
    Let $X$ be a non-empty set and let $A$ be a complete ortholattice. The \emph{full functional monadic ortholattice} determined by $X$ and $A$ is an algebra $\langle A^X; \cdot,+,-,c_0,c_1,\Diamond\rangle$ in which for all $f,g\in A^X$:
    \begin{enumerate}
        \item $(f\cdot g)(x)=f(x)\wedge g(x)$;
        \item $(f+g)(x)=f(x)\vee g(x)$; 
        \item $-f(x)=f(x)^{\perp}$; 
        \item $c_0(x)=0$ and $c_1(x)=1$; 
        \item $(\Diamond f)(x)=\bigvee\{f(x):x\in X\}$. 
    \end{enumerate}
\end{definition}
\begin{definition}\label{functional mol}
      An algebra is called a \emph{functional monadic ortholattice} if it is a subalgebra of a full functional monadic ortholattice. 
\end{definition}
\begin{proposition}\label{mol is a functional mol}
    The full functional monadic ortholattice generated by any non-empty set $X$ and any complete ortholattice $A$ is a monadic ortholattice. 
\end{proposition}
\begin{proof}
    It is obvious that $\langle A^X;\cdot,+,-,c_0,c_1\rangle$ forms a bounded complemented lattice since $\langle A;\wedge,\vee,^{\perp},0,1\rangle$ is a bounded complemented lattice, and hence we first verify that $-$ is an order-inverting involution. Assume $f(x)\leq g(x)$ so that $g(x)=(f+g)(x)=f(x)\vee g(x)$. It suffices to show that $-f(x)=(-f+-g)(x)$. Since $A$ is an ortholattice, our hypothesis implies: \[(-f+-g)(x)=-f(x)\vee-g(x)=f(x)^{\perp}\vee g(x)^{\perp}=f(x)^{\perp}=-f(x).\] It is obvious that $-$ is an involution since $--f(x)=f(x)^{\perp\perp}=f(x)$. 
    
    We now verify that $\Diamond$ is a quantifier on $\langle A^X;\cdot,+,-,c_0,c_1\rangle$. The proof that $\Diamond$ is an additive operation runs as follows: 
    \begin{align*}
        \Diamond (f+g)(x)&=\bigvee\{(f+g)(x):x\in X\}=\bigvee\{f(x)\vee g(x):x\in X\}
        \\&=\bigvee\{f(x):x\in X\}\vee\bigvee\{g(x):x\in X\}
        \\&=(\Diamond f)(x)\vee(\Diamond g)(x)=(\Diamond f+\Diamond g)(x)
    \end{align*}       
    Moreover, to see that $\Diamond$ is increasing, note that:
    \begin{align*}
        (\Diamond f+f)(x)&=(\Diamond f)(x)\vee f(x)=\bigvee\{f(x):x\in X\}\vee f(x)\\&=\bigvee\{f(x):x\in X\}=(\Diamond f)(x)   
    \end{align*}
    and hence $f(x)\leq(\Diamond f)(x)$. Moreover, we have: \[(\Diamond c_0)(x)=\bigvee\{c_0(x):x\in X\}=\bigvee\{0\}=0\] as well as $(\Diamond\Diamond f)(x)=\bigvee\{(\Diamond f)(x):x\in X\}=(\Diamond f)(x)$. It remains to check that $(\Diamond-\Diamond f)(x)=(-\Diamond f)(x)$, which is verified by the following: 
    \begin{align*}
        (\Diamond-\Diamond f)(x)& =\bigvee\{(-\Diamond f)(x):x\in X\}=\bigvee\{(\bigvee\{f(x):x\in X\})^{\perp}\} \\ &= (\bigvee\{f(x):x\in X\})^{\perp}= (-\Diamond f)(x). 
    \end{align*}
     Therefore, the full functional monadic ortholattice generated by $X$ and $A$ is a monadic ortholattice. 
\end{proof}
\begin{proposition}
    Every functional monadic ortholattice is a monadic ortholattice. 
\end{proposition}
\begin{proof}
    By Proposition \ref{mol is a functional mol}, for any set $X$ and complete ortholattice $A$, the full functional monadic ortholattice generated by $X$ and $A$ is a monadic ortholattice. Since monadic ortholattices form a variety, they are closed under subalgebras. Thus, since every functional monadic ortholattice is a subalgebra of a full functional monadic ortholattice, every functional monadic ortholattice is a monadic ortholattice.  
\end{proof}
To prove the functional representation of monadic ortholattices, we use the similar strategy inspired by \cite{bezhanishvili}.

\begin{definition}\label{super-amalgamation}
    A \emph{V-formation} in $\mathbb{OL}$ consists of a quintuple $\langle A,A_1,A_2,\phi_1,\phi_2\rangle$ such that $A$, $A_1$, and $A_2$ are ortholattices and $\phi_i\colon A\to A_i$ (for $i=1,2$) are ortholattice embeddings. An \emph{amalgamation} of a V-formation $\langle A,A_1,A_2,\phi_1,\phi_2\rangle$ in $\mathbb{OL}$ is a triple $\langle B,\psi_1,\psi_2\rangle$ such that $B$ is an ortholattice and $\psi_i\colon A_i\to B$ (for $i=1,2$) are ortholattice embeddings satisfying $\psi_1\circ\phi_1=\psi_2\circ\phi_2$ making the following diagram commute: 
    \begin{center}
    \begin{tikzcd}
                        & A_1 \arrow[rd, "\psi_1"] &   \\
A \arrow[rd, "\phi_2"'] \arrow[ru, "\phi_1"] &              & B \\
                        & A_2 \arrow[ru, "\psi_2"'] &  
\end{tikzcd}
\end{center}
This amalgamation is a \emph{super-amalgamation} if for all $a_1\in A_1$ and $a_2\in A_2$ with $1\leq i\not=k\leq 2$, $\psi_i(a_i)\leq\psi_k(a_k)$ implies there exists $a\in A$ satisfying: \[\psi_i(a_i)\leq\psi_i\circ\phi_i(a)=\psi_k\circ\phi_k(a)\leq\psi_k(a_k).\]        
\end{definition}
\begin{theorem}\label{ol has super amalgamation}
    Every $V$-formation in $\mathbb{OL}$ has a super-amalgamation. 
\end{theorem}
\begin{proof}
    It was shown by Bruns and Harding \cite{bruns} that $\mathbb{OL}$ satisfies the amalgamation property. Miyazaki \cite{miyazaki} extended their result by showing that $\mathbb{OL}$ satisfies the super-amalgamation property in the sense that every V-formation in $\mathbb{OL}$ admits of an super-amalgamation. 
\end{proof}
Recall that if $A$ is an ortholattice, its \emph{MacNeille completion} is a pair $\langle e;C\rangle$ such that $C$ is a complete lattice and $e\colon A\to C$ is a meet-dense and join-dense embedding (i.e., every element in $C$ is both a meet and a join of elements in the image of $A$ under $e$). A completion $\langle e;C\rangle$ is a \emph{regular completion} if $e$ preserves all existing meets and joins. 
\begin{theorem}\label{regular completion of ol}
    $\mathbb{OL}$ is closed under MacNeille completions. Moreover, every ortholattice admits of a regular completion.  
\end{theorem}
\begin{proof}
MacLaren \cite{maclaren} constructed the MacNeille completion of an ortholattice $A$ by first defining a relational structure $X_A=\langle A\setminus\{0\};\perp\rangle$ where $\perp\subseteq A\setminus\{0\}\times A\setminus\{0\}$ is defined by $a\perp b\Longleftrightarrow a\leq b^{\perp}$, and showed that $\perp$ is an orthogonality relation, i.e., irreflexive and symmetric, making $X_A$ an orthoframe. For any $U\subseteq X_A$, let $U^{\perp}=\{a\in A\setminus\{0\}:a\perp b\hspace{.2cm}\text{for all}\hspace{.2cm}b\in U\}$ and let $\mathcal{B}(X_A)=\{U\subseteq X_A:U=U^{\perp\perp}\}$. By Birkhoff \cite{birkhoff2}, $\langle\mathcal{B}(X);\cap,\sqcup,^{\perp},\emptyset,X\rangle$ is a complete ortholattice under $U\sqcup V:=(U\cup V)^{\perp\perp}$. Hence $\mathcal{B}(X_A)$ is a complete ortholattice. MacLaren showed that the map $g\colon A\to\mathcal{B}(X_A)$ defined by $g(a)={\downarrow}(a\setminus\{0\})$ provides a meet-dense and join-dense homomorphic embedding. Therefore, $\langle g;\mathcal{B}(X_A)\rangle$ is the MacNeille completion of $A$ and moreover, $\mathbb{OL}$ is closed under MacNeille completions since $\mathcal{B}(X_A)$ is an ortholattice. For part 2, it is known \cite{day} that the embedding associated with the MacNeille completion of any lattice preserves all existing meets and joins.  
\end{proof}

Throughout the remainder of this section, we shall fix a monadic ortholattice $A$ and set $B=\{a\in A:a=\exists a\}$. Note that $B$ is a sub-ortholattice of $A$. The following mimicks the recursive definition given by Bezhanishvili and Harding \cite{bezhanishvili} and relies on Theorem \ref{ol has super amalgamation}.  
\begin{definition}\label{recursive definition}
    Define recursively for each $n\geq 0$, ortholattices $A_n$, and injective ortholattice homomorphisms $f_n,g_n,h_n$ in the following manner: 
    \begin{enumerate}
        \item for the case when $n=0$, let $\langle A_0;f_0,g_0\rangle$ be a super-amalgamation of the V-formation $\langle B, A, A, 1_B, 1_B\rangle$ with ortholattice embeddings: 
        \[1_B\colon B\to A,\hspace{.2cm}f_0,g_0\colon A\to A_0\] satisfying the following equation: \[f_0\circ 1_A=g_0\circ 1_A\] so that the following diagram commutes: 
        
        \begin{center}
    \begin{tikzcd}
                        & A \arrow[rd, "f_0"] &   \\
B \arrow[rd, "1_B"'] \arrow[ru, "1_B"] &              & A_0 \\
                        & A \arrow[ru, "g_0"'] &  
\end{tikzcd}
\end{center}
        
        \noindent Then set: \[h_0=f_0|_B=g_0|_B.\]
        \item  for the case when $n>0$, let $\langle A_n,f_n,g_n\rangle$ be a super-amalgamation of the V-formation $\langle B,A_{n-1},A,h_{n-1},1_B\rangle$ with ortholattice embeddings:
        \[h_{n-1}\colon B\to A_{n-1}, \hspace{.2cm}1_B\colon B\to A, \hspace{.2cm} f_{n}\colon A_{n-1}\to A_n, \hspace{.2cm}g_n\colon A\to A_n\] satisfying the following equation: \[f_{n}\circ h_{n-1}=g_n\circ 1_B\] so that the following diagram commutes: 
         \begin{center}
    \begin{tikzcd}
                        & A_{n-1} \arrow[rd, "f_{n}"] &   \\
B \arrow[rd, "1_B"'] \arrow[ru, "h_{n-1}"] &              & A_n \\
                        & A \arrow[ru, "g_n"'] &  
\end{tikzcd}
\end{center}

        \noindent Then set: \[h_n=f_{n}\circ h_{n-1}=g_n|_B.\]   
    \end{enumerate}
\end{definition}

Now consider a directed family of injective ortholattice homomorphisms: 
\begin{center}
\begin{tikzcd}
A_0 \arrow[r, "f_1"] & A_1 \arrow[r, "f_2"] & A_2 \arrow[r, "f_3"] & ...
\end{tikzcd}
\end{center}
Let $C$ along with $(d_{n})_{n\in\omega}\colon A_n\to C$ be the directed limit of this family.     

The following will be used in the proof of the functional representation theorem and dualizes the arguments given in \cite[Lemma 3.3 and 3.4]{bezhanishvili}.  
\begin{lemma}\label{lemma 2.12}
  Let $B$ be given as in Definition \ref{recursive definition}. Then, for all $a\in B$ and $m,n\in\omega$, we have $d_m\circ g_m(a)=d_n\circ g_n(a)$. Therefore, for every $b\in A_m$, if $d_n\circ g_n(a)\leq d_m(b)$ for all $n\in\omega$, we have $d_n\circ g_n(\exists a)\leq d_m(b)$ for all $n\in\omega$.   
\end{lemma}
\begin{proof}
   We first note that it is obvious that $C$ is an ortholattice since $A_n$ is an ortholattice and that $f_n$ is an ortholattice embedding since each $d_n$ is an ortholattice embedding. For part 1, it suffices to show $d_n\circ g_n(a)=d_{n+1}\circ g_{n+1}(a)$ for all $a\in B$ and $n\in\omega$. First note that: 
   \begin{equation}\label{2.1}
       d_n\circ g_n(a)=d_n\circ h_n(a)
   \end{equation}
    since $h_n=g_n|_B$ by recursive clause 2 of Definition \ref{recursive definition}. Moreover, by the definition of direct limit we have $d_n=d_{n+1}\circ f_{n+1}$ and hence:
    \begin{equation}\label{2.2}
        d_n\circ h_n(a)=d_{n+1} \circ f_{n+1}\circ h_{n}(a).
    \end{equation}
 By recursive clause 2 of Definition \ref{recursive definition}, we have $g_{n}=f_n\circ h_{n-1}$. Therefore  $g_{n+1}|_B=f_{n+1}\circ h_n$, which yields the following:  
\begin{equation}\label{2.3}
    d_{n+1}\circ f_n\circ h_n(a)=d_{n+1}\circ g_{n+1}(a).
\end{equation}
By Equations \eqref{2.1}-\eqref{2.3}, we find that $d_n\circ g_n(a)=d_{n+1}\circ g_{n+1}(a)$, as desired. 

For part 2 of the claim, assume that $d_{m+1}\circ g_{m+1}(a)\leq d_{m}(b)$. It suffices to show that $d_{m+1}\circ g_{m+1}(\exists a)\leq d_m(b)$. Notice that by the definition of directed limit, we have $d_m=d_{m+1}\circ f_{m+1}$ and hence by our hypothesis: 
\begin{equation}\label{2.4}
    d_{m+1}\circ g_{m+1}(a)\leq d_{m+1}\circ f_{m+1}(b).
\end{equation}
Since $d_{m+1}$ is an ortholattice embedding, we have $g_{m+1}(a)\leq f_{m+1}(b)$. Observe that by Definition \ref{recursive definition}, the V-formation $\langle B;A_m,A,h_m,1_B\rangle$ is super-amalgamated by $\langle A_{m+1},f_{m+1},g_{m+1}\rangle$. Hence, there is some $c\in B$ such that:
\begin{equation}\label{2.5}
    g_{m+1}(a)\leq g_{m+1}\circ 1_B(c)=f_{m+1}\circ h_m(c)\leq f_{m+1}(b).
\end{equation}
 Clearly $g_{m+1}\circ 1_B(c)=g_{m+1}(c)$ and thus by (\ref{2.5}), $g_{m+1}(a)\leq g_{m+1}(c)$. Since $g_{m+1}$ is an ortholattice embedding, this implies $a\leq c$. Since $\exists$ is monotone, we have $\exists a\leq\exists c$. This, together with the fact that $B$ is the sub-ortholattice of closed elements of $A$ and $c\in B$, we have $\exists a\leq c$. Again, by (\ref{2.5}), we have $f_{m+1}\circ h_m(c)\leq f_{m+1}(b)$ and since $f_{m+1}$ is an ortholattice embedding, we have $h_m(c)\leq b$. Therefore, we have established $h_m(\exists a)\leq h_m(c)\leq b$ so in particular, we have $h_m(\exists a)\leq b$. This implies $d_m\circ h_m(\exists a)\leq d_m(b)$ but $h_m(\exists a)=g_m(\exists a)$ by recursive clause 2 and thus $d_m\circ g_m(\exists a)\leq d_m(b)$. Finally, by the first part of the proof, we have established that $d_m\circ g_m(a)=d_n\circ g_n(a)$ for all $m,n\in\omega$ and hence $d_n\circ g_n(\exists a)\leq d_m(b)$, as desired.   
\end{proof}
\begin{lemma}\label{lemma 2.13}
    Let $A$ be a monadic ortholattice and let $a\in A$. Then the set $\{d_n\circ g_n(a):n\in\omega\}$ has a least upper bound in $C$. 
\end{lemma}
\begin{proof}
    By part 1 of Lemma \ref{lemma 2.12}, we have for $k\in\omega$, $d_k\circ g_k(\exists a)=d_n\circ g_n(\exists a)$ for any $n\in\omega$. Then, since $\exists$ is increasing, we have $d_n\circ g_n(a)\leq d_k\circ g_k(\exists a)$. Thus $d_k\circ g_k(\exists a)$ is an upper bound of $\{d_n\circ g_n(a):n\in\omega\}$. Now let $b$ be any upper bound of $\{d_n\circ g_n(a):n\in\omega\}$ so that $d_n\circ g_n(a)\leq b$. By the definition of direct limits, there exists $m\in\omega$ and $c\in A_m$ such that $b=d_m(c)$. Then by part 2 of Lemma \ref{lemma 2.12}, we have $d_k\circ g_k(\exists a)\leq d_m(c)=b$, as desired.  
\end{proof}
The following is an immediate consequence of the above result. 
\begin{lemma}\label{cor1}
    For all $k\in\omega$, we have 
    $\bigvee\{d_n\circ g_n(a):n\in\omega\}=d_k\circ g_k(\exists a)$. 
\end{lemma}
The following result follows by similar arguments contained in the proofs of Lemma \ref{lemma 2.12} and Lemma \ref{lemma 2.13}. Although it is worth noting, we will not make use of it in the proof of our functional representation theorem. 
\begin{proposition}
    If $b\in A_m$ and $a\in A$, then if $d_m(b)\leq d_n\circ g_n(a)$ for all $n\in\omega$, then $d_m(b)\leq d_n\circ g_n(\forall a)$ for all $n\in\omega$. Thus $\{d_n\circ g_n(a):n\in\omega\}$ has a greatest lower bound. Thus for all $k\in\omega$, $\bigwedge\{d_n\circ g_n(a):n\in\omega\}=d_k\circ g_k(\forall a)$.  
\end{proposition}
\begin{proof}
    The claim follows from Lemmas \ref{lemma 2.12} and \ref{lemma 2.13}, the well-known fact that $\exists a\leq b\Longleftrightarrow a\leq\forall b$, and the inter-definability of $\exists$ and $\forall$ via the orthocomplementation operation.  
\end{proof}

\begin{theorem}\label{functional representation of mol}
    Every monadic ortholattice is isomorphic to a functional monadic ortholattice. 
\end{theorem}
\begin{proof}
    Let $A$ be a monadic ortholattice and let $C,A_n,g_n,d_n,$ be given as above. For simplicity, we identify $X$ with $\omega$ since $X$ is assumed to be countable.  Let $\bar{C}$ be a regular completion of $C$ whose associated ortholattice embedding is $i : C \to \bar{C}$. We know that $C$ admits of a regular completion by Theorem \ref{regular completion of ol}. Therefore, define: $$f : A \to \bar{C}^{\omega}\; \text{ where }  (a,n) \mapsto i \circ d_n \circ g_n (a). $$ We will show that $f$ determines a monadic ortholattice embedding from $A$ into the full functional monadic ortholattice $\langle(\bar{C})^{\omega};\Diamond\rangle$. Let $a \in A$, then: \begin{align*}
        f (\exists a)(n)& = i\circ d_n \circ g_n(\exists a) = i(\bigvee\{d_n \circ g_n(a):n \in \omega\}) \\
        & = \bigvee \{i \circ d_n \circ g_n (a):n \in \omega\} = \Diamond (i \circ d_n \circ g_n)(a) \\
        & = (\Diamond f(a))(n)
    \end{align*}
The second and third equality use Corollary \ref{cor1} and the fact that the embedding $i$ preserves meet operation. This shows that $f(\exists a)= \Diamond f(a)$ which proves that $f$ is a monadic ortholattice embedding. Therefore $A$ is isomorphic to a subalgebra of $\langle(\bar{C})^{\omega};\Diamond\rangle$ under $f$ and hence $A$ is isomorphic to a functional monadic ortholattice. 
\end{proof}
We proceed by extending Theorem \ref{functional representation of mol} to a functional representation theorem for locally finite $\delta$-free cylindric ortholattices, which form a special reducts of cylindric ortholattices. 

\begin{definition}\label{cylindric ortholattice}
   A \emph{cylindric ortholattice} is an ortholattice $A$ equipped with a family of unary operators $(\exists_i)_{i\in I}$ and a family of constants $(\delta_{i,k})_{i,k\in I}$ such that for all $i,k,l\in I$, the following conditions are satisfied:
    \begin{enumerate}
        \item $\exists_{i}\colon A\to A$ is a quantifier;
        \item $\exists_{i}\exists_{k}a=\exists_{k}\exists_{i}a$;
        \item $\delta_{i,k}=\delta_{k,i}$ and $\delta_{i,i}=1$;
        \item $i,l\not=k\Rightarrow\exists_{k}(\delta_{i,k}\wedge \delta_{k,l})=\delta_{i,l}$.
    \end{enumerate}
\end{definition}
\begin{remark}
Note that in the definition of cylindric algebras, there is an additional axiom asserting that for all $i,k\in I$ such that $i\not=k$, we have $\exists_i(d_{i,k}\wedge a)\wedge\exists_{k}(d_{i,k}\wedge a^{\perp})=0$. This induces an operation of substitution $\sigma^i_k\colon A\to A$ defined by $\sigma^i_k(a)=\exists_i(d_{i,k}\wedge a)$ that is an endomorphism of the Boolean algebra reduct. For details on why this axiom is not included in the definition of cylindric ortholattices, consult \cite[Remark 5.15]{harding1}.
\end{remark}
\begin{definition}
    A \emph{$\delta$-free cylindric ortholattice} is an ortholattice $A$ with a family of operators  $(\exists_{i})_{i\in I}$ satisfying conditions 1 and 2 of Definition \ref{cylindric ortholattice}. 
\end{definition}

In other words, a $\delta$-free cylindric ortholattice is an ortholattice with a family of pairwise commuting quantifiers. We note that in the case when $I=\{i\}$, a $\delta$-free cylindric ortholattice is precisely a monadic ortholattice. 
\begin{definition}
 Let $X$ be a non-empty set and let $A$ be a complete ortholattice. The \emph{full functional $\delta$-free cylindric ortholattice} determined by $X$ and $A$ is an algebra $\langle A^X; \cdot,+,-,c_0,c_1,(\Diamond_i)_{i\in I}\rangle$ such that $\langle A^X; \cdot,+,-,c_0,c_1,\Diamond_i\rangle$ is a full functional monadic ortholattice for each $i\in I$.   
\end{definition}
\begin{definition}
    An algebra is a \emph{functional $\delta$-free cylindric ortholattice} if it is a subalgebra of a full functional $\delta$-free cylindric ortholattice.
\end{definition}
\begin{proposition}\label{diag free in a diag free ol}
    Let $X$ be a non-empty set and let $A$ be a complete ortholattice. The full functional $\delta$-free cylindric ortholattice determined by $X$ and $A$ is a $\delta$-free cylindric ortholattice.  
    \end{proposition}
    \begin{proof}
By Proposition \ref{mol is a functional mol}, we know that the full functional monadic ortholattice $\langle A^X; \cdot,+,-,c_0,c_1,\Diamond_i\rangle$ is a monadic ortholattice for each $i\in I$. Therefore, it remains to verify that $(\Diamond_i)_{i\in I}$ is a family of pairwise commuting quantifiers on $\langle A^X; \cdot,+,-,c_0,c_1\rangle$, but this is obvious by the definition of $\Diamond_i$ since: \[(\Diamond_i\Diamond_kf)(x)=\bigvee\{(\Diamond_kf)(x):x\in X\}=\bigvee\{(\Diamond_if)(x):x\in X\}=(\Diamond_k\Diamond_if)(x).\]    Therefore, the full functional $\delta$-free cylindric ortholattice determined by $X$ and $A$ is a $\delta$-free cylindric ortholattice.         
    \end{proof}

The following is immediate by Theorem \ref{functional representation of mol} and Proposition \ref{diag free in a diag free ol}. 
\begin{corollary}\label{functional rep of diag free}
    Every $\delta$-free cylindric ortholattice is isomorphic to a functional $\delta$-free cylindric ortholattice. 
\end{corollary}

We note that cylindric algebras provide algebraic models for the full classical predicate calculus with equality, and hence their $\delta$-free cylindric Boolean algebra reducts provide algebraic models for the classical predicate calculus without equality. Analogously, whereas cylindric ortholattices provide algebraic models for full predicate orthologic with equality, $\delta$-free cylindric ortholattices provide algebraic models for predicate orthologic without equality. Consult \cite{chiara} for more details pertaining to predicate orthologic.  
\section{Functional locally finite $\sigma$-free polyadic ortholattices}
In this section, we introduce locally finite substitution-free ($\sigma$-free) polyadic ortholattices and show that they stand in a one-to-one correspondence with the locally finite $\delta$-free cylindric ortholattices. Through this correspondence, along with the results obtained in the previous section, we show that every locally finite $\sigma$-free polyadic ortholattice is isomorphic to a functional one.  

        \begin{definition}\label{substitution free polyadic ortholattice}
        A \emph{$\sigma$-free polyadic ortholattice} consists of an ortholattice $A$ and a family of operations $(\nabla_J)_{J\subseteq I} : \wp(I) \to \text{Quant}(A)$, where $\text{Quant}(A)$ is the collection of all quantifiers on $A$, satisfying the following conditions: 
        \begin{enumerate}
            \item $\nabla_{\emptyset}a = a$ for all $a \in A$;
            \item $\nabla_{J \cup K}a= \nabla_J \circ\nabla_Ka$ for all non-empty $J,K\subseteq I$.
        \end{enumerate}
    \end{definition}  

    \begin{definition}
        For a $\sigma$-free polyadic ortholattice $\langle A,(\nabla_J)_{J\subseteq I} \rangle $, we call a subset $J$ of I the {\it support} of an element $a$ whenever $\nabla_{I\setminus J}a=a$. Then, we say that a $\sigma$-free polyadic ortholattice is \textit{locally finite} if every element $a \in A$ has a finite support.  
    \end{definition}
 
The following relation is needed in the construction of a full functional $\sigma$-free polyadic ortholattice. 

\begin{definition}
    Given two sets $I,J$ with $J \subseteq I$ and a non-empty set $X$. For any $\vec{x},\vec{y}\in X^I $, we define the following relation $J_*\subseteq X^I\times X^I$ as follows:
\[\vec{x}J_*\vec{y}\Longleftrightarrow\vec{x}(i)=\vec{y}(i)\]    
for all $i \in I\setminus J$. 
\end{definition}

\begin{definition}
Let $I$ be a set and $X$ be a non-empty sets and let $A$ be a complete ortholattice. A \emph{full functional $\sigma$-free polyadic ortholattice} determined by $I$, $X$, and $A$ is an algebraic structure $\langle A^{X^I};\cdot,+,-,c_0,c_1,(\widehat{\nabla_J})_{J \subseteq I}\rangle$ such that $f,g\in A^{X^I}$ and $\vec{x} \in X^I$, the following conditions are satisfied:
\begin{enumerate}
    \item $(f\cdot g)(\vec{x})= f(\vec{x})\wedge g(\vec{x})$;
    \item $(f+g)(\vec{x})=f(\vec{x})\vee g(\vec{x})$;
    \item $-f(\vec{x})=f(\vec{x})^{\perp}$;
    \item $c_0(\vec{x})=0$ and $c_1(\vec{x})=1$;
    \item $\widehat{\nabla_J}f(\vec{x})=\bigvee\{f(\vec{y}):\vec{x}J_*\vec{y}\;\}$. 
\end{enumerate}
\end{definition}

The following describes the conditions under which a full functional $\sigma$-free polyadic ortholattice is locally finite. 
\begin{definition}
For a full functional $\sigma$-free polyadic  ortholattice, we say $f(\vec{x})$ is \emph{independent} of a subset $J $ of $ I$ if $f(\vec{x})=f(\vec{y})$ for all $\vec{x} J_* \vec{y}$. Then, we say that a full functional $\sigma$-free polyadic ortholattice is \textit{locally finite} if every $f(\vec{x})$ is independent of $I\setminus J$ for a cofinite subset $J$ of $I$.
\end{definition}
\begin{definition}\label{funcitonal substitution free}
    A \emph{functional locally finite $\sigma$-free polyadic ortholattice} is a subalgebra of a full functional locally finite $\sigma$-free polyadic ortholattice. 
\end{definition}
\begin{proposition}\label{full function substitution free}
    For any sets $I$, non-empty sets $X$, and complete ortholattices $A$, the full functional locally finite $\sigma$-free polyadic ortholattice determined by $I$, $X$, and $A$ is a locally finite $\sigma$-free polyadic ortholattice.  
\end{proposition}
\begin{proof}
    It is obvious that $\langle A^{X^I};\cdot,+,-,c_0,c_1\rangle$ is an ortholattice since $A$ is an ortholattice. Hence, we first verify that $\widehat{\nabla_J}$ satisfies conditions 1 and 2 of Definition \ref{substitution free polyadic ortholattice}. For condition 1, note that when $J=\emptyset$, it is clear that: \[\widehat{\nabla}_Jf(\vec{x})=\bigvee\{f(\vec{y}):\vec{x}J_*\vec{y}\}=f(\vec{x}).\] Now choose $J,K\subseteq I$ with $J,K\not=\emptyset$. Then, by the definition of $\widehat{\nabla}$, we have: \[\widehat{\nabla_{J\cup K}}f(\vec{x})=\bigvee\{f(\vec{y}):\vec{x}(J\cup K)_*\vec{y}\}.\] Note that by the definition of $(J\cup K)_*$, we have $\vec{x}(J\cup K)_*\vec{y}$ if and only if $\vec{x}(i)=\vec{y}(i)$ for all $i\in U\cup K$. This, along with the definition of $\widehat{\nabla}$, then yields:  
\begin{align*}
    \bigvee\{f(\vec{y}):\vec{x}(J\cup K)_*\vec{y}\}&=\bigvee\{f(\vec{y}):x(i)=y(i)\hspace{.2cm}\text{for all}\hspace{.2cm}i\in I\setminus(J\cup K)\}\\&=\widehat{\nabla_J}\bigvee\{f(\vec{y}):x(i)=y(i)\hspace{.2cm}\text{for all}\hspace{.2cm}i\in I\setminus K\}
    \\&=\widehat{\nabla_J}\circ\widehat{\nabla_K}f(\vec{x})
\end{align*}
Lastly, we verify that the local finiteness property is satisfied. Hence, suppose that $f(\vec{x})$ is independent of $I\setminus J$ for a cofinite subset $J$ of $I$. Then: \[\widehat{\nabla}_{I\setminus J}f(\vec{x})=\bigvee\{f(\vec{y}):\vec{x}(I\setminus J)_*\vec{y}\}=f(\vec{x})\] which show that J is a finite support of $f(\vec{x})$.
\end{proof}
The following is immediate by Proposition \ref{full function substitution free} and Definition \ref{funcitonal substitution free}.   
\begin{proposition}
    Every functional locally-finite $\sigma$-free polyadic ortholattice is a locally finite $\sigma$-free polyadic ortholattice. 
\end{proposition}

Analoglously to that of $\sigma$-free polyadic ortholattices, the $\delta$-free cylindric ortholattices considered in the previous section admit of a natural notion of local finiteness.

\begin{definition}
        A $\delta$-free cylindric ortholattice is called \textit{locally finite} if for each $a \in A$, the set $\{j \in I : \exists_j a=a  \}$ is cofinite. 
 \end{definition}

 Following the correspondence between cylindric algebras and polyadic algebras with equality established by Galler \cite{galler1}, we proceed by analogously demonstrating that locally finite $\delta$-free cylindric ortholattices correspond to locally finite $\sigma$-free polyadic ortholattices.

\begin{lemma}\label{col to pol}
    Every locally finite $\delta$-free cylindric ortholattice induces a locally finite $\sigma$-free polyadic ortholattice. 
\end{lemma}
\begin{proof}
    Suppose $\langle A; (\exists_i)_{i\in I}\rangle$ is a locally finite $\delta$-free cylindric ortholattice. For $J \subseteq I$, we define $\nabla_J$ as $\nabla_{\emptyset}a=a, \nabla_{\{j\}}a=\exists_ja$, and if $J=\{j_1,\cdots,j_n \}$ is a finite, we then define $\nabla_Ja$ by the equation $\nabla_Ja=\exists_{j_1}\cdots \exists_{j_n}a$. For an infinite set $J$, we define $\nabla_J a$ as follows : consider the set $S_a = \{j\in I:\exists_ja=a\}$ which is cofinite and $\nabla_Ja :=\nabla_{J\setminus(J\cap S_a)} \nabla_{J\cap S_a}a = \nabla_{J\setminus S_a} a$ which can be defined as in the finite case. To verify the axioms,  we have 
    \begin{align*}
        \nabla_{J\cup K}a&= \nabla_{(J \cup K) \setminus  S_a}a= \exists_{j_1}\dots\exists_{j_n}\exists_{k_1}\dots\exists_{k_n}a\\&=\nabla_{J\setminus S_a} \circ\nabla_{K\setminus S_a} a= \nabla_J \circ\nabla_K a .
    \end{align*}
 For each $a \in A$, the set $I\setminus S_a$ is a finite support of $a$ since $\nabla_{S_a}a =a$.
\end{proof}
\begin{lemma}\label{pol to col}
    Every locally finite $\sigma$-free polyadic ortholattice induces a locally finite $\delta$-free cylindric ortholattice. 
\end{lemma}
\begin{proof}
   Suppose that  $\langle A; \wedge, \vee, ^{\perp},0,1,(\nabla_J)_{J\subseteq I}\rangle$ be a $\sigma$-free polyadic ortholattice.  Define $\exists_i a=\nabla_{\{i\}}a$ for all $i \in I$. It suffices to verify the commutativity of $\exists_i$ and the local finiteness property. Let $a\in A$, then we have:   
       $$\exists_i\circ\exists_k a  =\nabla_{\{i\}}\circ\nabla_{\{k\}}a = \nabla_{\{i\}\cup\{k\}}a = \nabla_{\{k\}}\circ\nabla_{\{i\}} a = \exists_k\circ\exists_i a.$$
For local finiteness, we need to show that $S_a:= \{j\in I:\exists_ja=a\}$ is cofinite for each $a \in A$. Fix $a\in A$ and a finite set $J$ to be a support of $a$. If $j \in I\setminus J$: \[\exists_ja=\nabla_{\{j\}}a=\nabla_{\{j\}}\circ\nabla_{I\setminus J}a=\nabla_{I\setminus J}a =a\] which implies $j \in S_a$. Hence $I\setminus S_a \subseteq J $ which implies $S_a$ is cofinite.
\end{proof}
\begin{theorem}\label{bijection}
    There exists a one-to-one correspondence between locally finite $\delta$-free cylindric ortholattices and locally finite $\sigma$-free polyadic ortholattices. 
\end{theorem}
\begin{proof}
Let $\mathfrak{A}=\langle A;(\exists_i)_{i\in I}\rangle$ be a locally finite $\delta$-free cylindric ortholattice, let $\mathfrak{B}=\langle A;(\nabla_J)_{J\subseteq I}\rangle$ be the locally finite $\sigma$-free polyadic ortholattice obtained from $\mathfrak{A}$ via Lemma \ref{col to pol}, and let $\widehat{\mathfrak{A}}=\langle A;(\widehat{\exists_i})_{i\in I}\rangle$ be the locally finite $\delta$-free cylindric ortholattice obtained from $\mathfrak{B}$ via Lemma \ref{pol to col}. Clearly Lemmas \ref{col to pol} and \ref{pol to col} yield $\exists_ja=\nabla_Ja=\widehat{\exists_j}a$ for all $a\in A$.          
\end{proof}
We now arrive at the promised functional representation result for locally finite $\sigma$-free polyadic ortholattices. 
\begin{theorem}
 Every locally finite $\sigma$-free polyadic ortholattice is isomorphic to a functional one. 
\end{theorem}
\begin{proof}
   The result follows immediately by Theorem \ref{functional representation of mol}, Corollary \ref{functional rep of diag free}, and Theorem \ref{bijection}. 
\end{proof}
\section{Conclusions and future work}

The primary contribution of this paper was to prove a functional representation theorem for monadic ortholattices. This was achieved by employing the methods used in \cite{bezhanishvili} by working with super-amalgamations and regular completions. As a consequence of this result, we then provided a functional representation of $\delta$-free cylindric ortholattices. We then introduced locally finite $\sigma$-free polyadic ortholattices and showed that they stand in a one-to-one correspondence with the locally finite $\delta$-free cylindric ortholattices. Under this correspondence, along with our functional representation result for $\delta$-free cylindric ortholattices, we demonstrated that every locally finite $\sigma$-free polyadic ortholattice is isomorphic to a functional one.

Future lines of work include extending the correspondence established in Theorem \ref{bijection} to the case of cylindric ortholattices and polyadic ortholattices with equality. In addition, it would be interesting to investigate the functional representation theorems established in this work from the perspective of relation algebras \cite{hodkinson1}. Moreover, we reiterate the question posed in \cite{harding1} of determining whether every quantum monadic algebra, i.e., an orthomodular lattice equipped with a quantifier, is isomorphic to a functional one. Since the MacNeille completion of an orthomodular lattice is not necessarily orthomodular \cite{harding3} and orthomodular lattices do not satisfy the amalgamation property \cite{bruns}, this will likely pose some technical challenges.  

\section*{Acknowledgments}
\subsection*{Funding}
The first author was supported under the grant 22-23022L CELIA of Grantová Agentura České Republiky and the Department of Logic SVV grant no. 260677. The second author was supported under the Research Council of Canada grant no.~767-2022-1514 (SSHRC-CGSD) and the University of Alberta Graduate Student Research Travel Grant.

\section*{Declarations}

\subsection*{Data availability}

Data sharing not applicable to this article as datasets were neither generated
nor analyzed.

\subsection*{Compliance with ethical standards}

The authors declare that they have no conflict of interest.

\subsection*{Author's contributions}
Both authors contributed equally.

\end{document}